\documentclass[]{amsart}
\usepackage{amsfonts}
\usepackage{indentfirst,amsthm,bm,amsmath}
\usepackage{hyperref}
\makeatletter
\date{}
\newtheorem{theorem}{Theorem}[section]
\newtheorem{corollary}{Corollary}[section]
\newtheorem{lemma}{Lemma}[section]

\newtheorem{remark}{Remark}[section]
\newtheorem{example}{Example}[section]
\newtheorem{question}{Question}[section]
\begin{document}

\title[]{Logarithmic difference lemma in several complex variables and partial difference equations}

\author[Tingbin Cao]{Tingbin Cao}
\address[Tingbin Cao]{Department of Mathematics, Nanchang University, Nanchang, Jiangxi 330031, P. R. China}
\email{tbcao@ncu.edu.cn}
\thanks{The work is supported by the National Natural Science Foundation of China (\#11871260, \#11461042) and the outstanding young talent assistance program of Jiangxi Province (\#20171BCB23002) in China.}

\author[Ling Xu]{Ling Xu}
\address[Ling Xu]{Department of Mathematics, Nanchang University, Nanchang, Jiangxi 330031, P. R. China}
\email{xuling-jxstnu@foxmail.com}





\begin{abstract} In this paper, we mainly propose improvements of the logarithmic difference lemma for meromorphic functions in several complex variables, and then investigate meromorphic solutions of partial difference equations from the viewpoint of Nevanlinna theory.
\end{abstract}

\maketitle

\section{Introduction}

It is well known that the celebrated binomial function $C^{n}_{m}=\frac{n!}{(n-m)!m!}$ $(1\leq m\leq n)$ having the relation \begin{eqnarray}\label{E1.1}C^{n}_{m}=C^{n-1}_{m-1}+C^{n-1}_{m}\end{eqnarray}
in the early history of mathematics, which was known to Shijie Zhu in China in 1303. The functional relation \eqref{E1.1} is an example of partial difference equations which while was developed only after the 18th century. Discrete analogs of equations of mathematical physics have always been of great interest to scholars. For instances, R. Courant, K. Friedrichs and H. Lewy \cite{courant} discussed  algebraic problems of a very much simpler structure by replacing the differentials by difference quotients on some (say rectilinear) mesh. Y. Y. Azmy and V. Protopopescu \cite{azmy-protopopescu} investigated various aspects of the dynamics of a discrete reaction-diffusion system. D. Young \cite{young} introduced iterative methods to solve partial difference equations.  Although partial difference equations such as \eqref{E1.1} appear well before partial differential equations, it has not drawn as much attention as their continuous counterparts. Renewed interest has, however, been picking up momentum during the last sixty years among mathematician, physicists, engineers and computer scientists. For many examples of partial difference equations and its background, we refer to see \cite{Cheng-book, davis, duffin-rohrer}. If the continuous counterparts of \eqref{E1.1} are considered, then we find an interesting phenomenon that the entire function $f(z_{1}, z_{2})=e^{z_{1}+z_{2}}$ on $\mathbb{C}^{2}$ is a nontrivial solution of the partial difference equation $$f(z_{1}, z_{2})=f(z+c_{1}, z_{2}+c_{2})+f(z_{1}, z_{2}+c_{2}),$$ where $c_{1}, c_{2}$ are values in $\mathbb{C}^{2}$ such that $e^{c_{1}+c_{2}}+e^{c_{2}}=1.$  Motivated by this, it is worth in considering entire or meromorphic solutions of partial difference equations.\par

As early as over 30 years ago, several initial results on the existence of meromorphic solutions of some complex difference equations have been obtained by Bank, Kaufman, Shimomura, Yanagihara  and other researchers. Later on , the researches in this field were developed slowly, almost in a state of stagnation. Until recent ten years, Nananlinna theory (especially the difference analogues such as logarithmic derivative lemma, Tumura-Clunie theorem etc.)  has been used as a powerful tool to investigate complex difference equations, and thus it becomes an interesting and hot direction. For this background, we refer to see \cite{halburd-korhonen-1, chiang-feng, chen-book, zheng-korhonen-2018}.\par

As far as we know, however, there are very little of results on solutions of complex partial difference equations by using Nevanlinna theory. In 2012, Korhonen \cite{korhonen-1} firstly obtained the difference version of logarithmic derivative lemma (shortly, we may say logarithmic difference lemma) for meromorphic functions on $\mathbb{C}^{m}$ with hyperorder strictly less than $\frac{2}{3},$ and then used it to consider a class of partial difference equations in the same paper. In \cite{cao-korhonen-2016}, Cao and Korhonen improved the logarithmic difference lemma to the case where the hyperorder is strictly less than one. Meanwhile, Wang \cite{wangyue} considered some kinds of partial $q$-difference equations. \par

The main purpose of this paper is to improve the logarithmic difference lemma in Nevanlinna theory and use it to study complex partial difference equations, basically focus on linear partial difference equations, nonlinear partial difference equations, difference counterpart of Tumura-Clunie theorem concerning partial difference equations. We first introduce some basic notations and definitions as follows. Let $z=(z_{1}, \ldots, z_{m})\in\mathbb{C}^{m}$ with $\|z\|^{2}=\sum_{j=1}^{m}|z_{j}|^{2}.$ Define the differential operators $d=\partial+\overline{\partial}$ and $d^{c}=\frac{\partial-\overline{\partial}}{4\pi i }.$ For a meromorphic function $f$ on $\mathbb{C}^{m},$ let $\nu_{f-a}^{0}$ be the zero divisor of $f-a.$ Set $n(t, \frac{1}{f-a})=\int_{supp \nu_{f-a}^{0}\cap B_{m}(t)}\nu_{f-a}^{0}(z) (dd^{c}\|z\|^{2})^{m-1}$ if $m\geq 2;$ and $n(t, \frac{1}{f-a})=\sum_{|z|\leq t}\nu_{f-a}^{0}(z)$ if $m=1,$  where $B_{m}(t)=\{z: \|z\|\leq t\}.$  Denote by $N(r, \frac{1}{f-a})=\int_{1}^{r}\frac{n(r,\frac{1}{f-a})dt}{t}$ the counting functions of zeros of $f-a$ on complex vector space $\mathbb{C}^{m},$  by $m(r, f)$ the proximity function of $f$ defined as
$m(r,f)=\int_{\partial B_{m}(r)}\log^{+}\left|f(z)\right|\sigma_{m}(z)$ where $\sigma_{m}(z)=d^{c}\log\|z\|^{2}\wedge (dd^{c}\|z\|^{2})^{m-1}$ and $\log^{+} x=\max\{\log x, 0\}.$ Then the Nevanlinna characteristic function of $f$ is defined as $T(r,f)=N(r, f)+m(r,f).$ Then the first main theorem is said that $$T(r, \frac{1}{f-a})=T(r,f)+O(1)$$ for any value $a\in \mathbb{C}\cup\{\infty\}.$ A meromorphic function $f$ can be also seen as a holomorphic curve from $\mathbb{C}^{m}$ into $\mathbb{P}^{1}(\mathbb{C})$ with a reduced representation $f=(f_{0}, f_{1}),$ where $f_{0}$ and $f_{1}$ are entire function on $\mathbb{C}^{m}$ without common zeros. The Cartan characteristic function is defined by $T_{f}(r)=\int_{\partial B_{m}(r)}\log\max\{|f_{0}(z)|, |f_{1}(z)|\}\sigma_{m}(z)-\int_{\partial B_{m}(1)}\log\max\{|f_{0}(z)|, |f_{1}(z)|\}\sigma_{m}(z).$ The two characteristic functions have the relation $T_{f}(r)=T(r, f)+O(1).$ The defect $\delta_{f}(a)$ of zeros of $f-a$ is defined as $$\delta_{f}(a)=1-\limsup_{r\rightarrow\infty}\frac{N(r, \frac{1}{f-a})}{T(r,f)}.$$ The order $\rho(f)$ and hyperorder $\rho_{2}(f)$ of $f$ are defined respectively by
$$\rho(f)=\limsup_{r\rightarrow\infty}\frac{\log T(r,f)}{\log r},$$ and
$$\rho_{2}(f)=\limsup_{r\rightarrow\infty}\frac{\log\log T(r,f)}{\log r}.$$ We assume that the readers are familiar with the basic notations and results on Nevanlinna theory for meromorphic functions in several complex variables (refer to see, for examples \cite{griffiths, ru-1, noguchi-winkelmann}).\par

The logarithmic difference lemma of several complex variables in Nevanlinna theory will play the key role in studying meromorphic solutions of complex partial difference equations, as does as the logarithmic derivative lemma of several complex variables in investigating solutions of complex partial differential equations \cite{libaoqin-1, libaoqin-2, hu-yang-1}. Let $c\in\mathbb{C}^{m}\setminus\{0\}.$ Motivated by the ideas of \cite{cao-zheng-2018, zheng-korhonen-2018}, we continue to propose an improvement of the logarithmic difference lemma for meromorphic functions in several complex variables \cite{korhonen-1, cao-korhonen-2016} (Theorem \ref{T1}) that \begin{eqnarray}
m\left(r, \frac{f(z+c)}{f(z)}\right)=o\left(T(r, f)\right)
\end{eqnarray} holds for all $r$ possible outside of a set $E$ with zero upper density measure, provided that the growth of the meromorphic function $f$ on $\mathbb{C}^{m}$ satisfies \begin{equation}\label{E1.3}\limsup_{r\rightarrow\infty}\frac{\log T(r, f)}{r}=0\end{equation} (which implies that the hyperorder is rather than just strictly less than one). This is also an extension of \cite{halburd-korhonen-1, halburd-korhonen-tohge-1, zheng-korhonen-2018} from one variable to several variables. Then from it, we get the relation \begin{equation}T(r, f(z+c))=T(r, f)+o(T(r, f)), \,\,\,\, (r\not\in E),\end{equation}
under the assumption of \eqref{E1.3}.  We will also show the explicit expression of $o(T(r,f))$ in the logarithmic difference lemma for the special case whenever $f$ is of finite order as (Theorem \ref{T1'})
\begin{eqnarray}
m\left(r, \frac{f(z+c)}{f(z)}\right)=O\left(r^{\rho(f)-1+\varepsilon}\right),
\end{eqnarray} and  thus obtain the relation \begin{equation}T(r, f(z+c))=T(r, f)+O(r^{\rho(f)-1+\varepsilon})\end{equation} for any $\varepsilon(>0).$ This is an extension of Chiang and Feng \cite{chiang-feng} from one variable to several variables.\par

In terms of the above results on the logarithmic difference lemma, we can consider meromorphic solutions of partial difference equations. Since there are too many kinds of partial difference equations, we can not systematically and completely investigate solutions of partial difference equations. In this paper, we will focus on some typical models of partial difference equations. For the discrete potential Korteweg-de Vries (KdV) equations $$X^{i+1}_{j+1}=X_{j}^{i}+\frac{Z_{j}^{i}}{X_{j}^{i+1}-X_{j+1}^{i}}$$ in \cite{tremblay-grammaticos-ramani}, we can firstly consider the nonlinear partial difference equation
\begin{eqnarray}\label{E1.2}f(z_{1}+c_{1}, z_{2}+c_{2})=f(z_{1}, z_{2})+\frac{A(z_{1}, z_{2})}{f(z_{1}, z_{2}+c_{2})-f(z_{1}+c_{1}, z_{2})}\end{eqnarray}
where $c_{1}, c_{2}\in\mathbb{C}\setminus\{0\},$ and $A(z_{1}, z_{2})$ is a nonzero meromorphic function on $\mathbb{C}^{2}$ such that $T(r, A)=o(T(r, f))$ (or say, $A$ is a small function with respect to $f$). In fact, we will obtain  (Theorem \ref{T10'}) that \emph{any nontrivial meromorphic solution of the equation \eqref{E1.2} with the assumption \eqref{E1.3} must satisfy $\delta_{f}(0)>0.$} We also consider the Fermat type nonlinear partial difference equation
\begin{eqnarray*}\frac{1}{f^{m}(z_{1}+c_{1}, z_{2}+c_{2})}+\frac{1}{f^{m}(z_{1}, z_{2})}=A(z_{1}, z_{2})f^{n}(z_{1}, z_{2}),\end{eqnarray*} or
\begin{eqnarray*}\frac{1}{f^{m}(z_{1}+c_{1}, z_{2}+c_{2})}+\frac{1}{f^{m}(z_{1}+c_{1}, z_{2})}+\frac{1}{f^{m}(z_{1}, z_{2}+c_{2})}=A(z_{1}, z_{2})f^{n}(z_{1}, z_{2}),\end{eqnarray*}  and prove that any nontrivial meromorphic solution $f$ with $\delta_{f}(\infty)>0$ satisfies $\limsup_{r\rightarrow\infty}\frac{\log T(r, f)}{r}>0,$ provided that $A$ is a small function with respect to $f$  (see Theorem \ref{TFE}). Furthermore, we will prove the difference versions of the well-known  Tumura-Clunie theorem in several complex variables which is a powerful tool for studying complex (partial) differential equations (see for examples \cite{laine, hu-yang, hu-yang-1, libaoqin}). \par

There are many models of partial linear difference equations (see \cite{Cheng-book}), such as the two-level discrete heat equation
$$u_{j+1}^{i}=au^{i}_{j-1}+bu^{i}_{j}+cu^{i}_{j+1},$$
the nonsymmetric partial difference functional equation
$$\frac{u_{x+t, y}-2u_{x, y}+u_{x-t, y}}{t^{2}}=\frac{u_{x, y+s}-2u_{x, y}+u_{x,y-s}}{s^{2}},$$
and the steady state discrete Laplace equation
$$u_{m-1,n}+v_{m+1, n}+u_{m, n-1}+u_{m, n+1}-4u_{m,n}=0.$$
These equations impel us to study general linear homogeneous partial difference equations \begin{eqnarray}\label{E1.4}
A_{n}(z)f(z+c_{n})+\ldots+A_{1}(z)f(z+c_{1})+A_{0}(z)f(z)=0,
\end{eqnarray} where $A_{0}, \ldots, A_{n}$ are meromorphic functions on $\mathbb{C}^{m}$ and $c_{1}, \ldots, c_{n}\in\mathbb{C}^{m}\setminus\{0\}.$ According to the logarithmic difference lemma for finite order, we will obtain (Theorem \ref{T7}) that \emph{any nontrivial meromorphic solution $f$ of \eqref{E1.4} satisfies $\rho(f)\geq\rho(A_{k})+1,$ whenever one transcendental meromorphic coefficient $A_{k}$ $(k\in\{0, 1, \ldots, n\})$ dominates the growths of all the meromorphic coefficients.}  Motivated by the the model of the discrete or finite Poisson equation (see \cite{Cheng-book}) $$u_{i, j+1}+u_{i+1, j}+u_{i, j-1}+u_{i-1,j}-4u_{i,j}=g_{ij},$$ we also consider the linear nonhomogeneous partial difference equations \begin{eqnarray}\label{E18}
A_{n}(z)f(z+c_{n})+\ldots+A_{1}(z)f(z+c_{1})+A_{0}(z)f(z)=F(z),
\end{eqnarray}where meromorphic coefficients $A_{0}, \ldots, A_{n}, F(\not\equiv 0)$ on $\mathbb{C}^{m}$ are small functions with respect to meromorphic solutions $f.$ We will prove (Theorem \ref{T8}) that \emph{if a meromorphic solution $f$ of \eqref{E18} satisfies the assumption of $\limsup_{r\rightarrow\infty}\frac{\log T(r, f)}{r}=0,$ then we have $\delta_{f}(0)=0.$}
\par

This paper is organized as follows. Three forms of the logarithmic difference lemma for meromorphic functions in several complex variables are proved in Section \ref{s2}. By them, the relations of $N(r, f)\sim N(r, f(z+c))$ and $T(r, f)\sim T(r, f(z+c))$ are given in the same section. In Section \ref{s3}, we firstly consider nonlinear partial difference equations coming from the discrete potential KdV equation and the Fermat equation, and then study general partial linear difference equations. Difference analogues of Tumura-Clunie theorem concerning partial difference polynomials are also investigated in Section \ref{s3}.  Finally, we obtain an improvement of Korhonen's result for a class of complex partial difference equations by our logarithmic difference lemma. Some examples are given to show that the results of nonlinear partial difference equations or the linear partial difference equations are sharp.

\section{Logarithmic difference lemma in several complex variables}\label{s2}

In this section, to solve meromorphic solutions of partial difference equations, we mainly study the logarithmic difference lemma of several complex variables of Nevanlinna theory. In 2006, Halburd-Korhonen \cite[Theorem 2.1]{halburd-korhonen-1} and Chiang-Feng \cite{chiang-feng} obtained independently the difference version of logarithmic derivative lemma (shortly say, logarithmic difference lemma) for meromorphic functions with finite order on the complex plane. In 2014, Halburd, Korhonen and Tohge \cite[Theorem 5.1]{halburd-korhonen-tohge-1} extended it to the case for hyperorder strictly less than one. In the high dimensional case,  Korhonen \cite[Theorem 3.1]{korhonen-1} gave a logarithmic difference lemma for meromorphic functions in several variables of hyperorder strictly less that $2/3$. In 2016, Cao and Korhonen \cite{cao-korhonen-2016} improved it to the case for meromorphic functions with hyperorder $<1$ in several variables. Very recent, Zheng and Korhonen \cite{zheng-korhonen-2018} improve the condition to the case when the meromorphic funtion $f$ on the plane satisfies $\limsup_{r\rightarrow\infty}\frac{\log T(r, f)}{r}=0$ (rather than just hyperorder strictly less than one) which is usually called minimal type. In fact, they proved a version of the subharmonic functions for the logarithmic difference lemma. Here, we improve and extend the known results on logarithmic difference lemma directly for meromorphic fucntions of one and several complex variables by using a growth lemma for nondecreasing positive logarithmic convex function due to Zheng and Korhonen, but avoiding the subharmonic function theory. A tropical version of logarithmic derivative lemma due to Cao and Zheng \cite{cao-zheng-2018} was obtained very recently.\par

\begin{theorem}\label{T1}
Let $f$ be a nonconstant meromorphic function on $\mathbb{C}^{n},$ and let $c\in\mathbb{C}^{n}\setminus\{0\}.$ If \begin{eqnarray}\label{E1}\limsup_{r\rightarrow\infty}\frac{\log T(r, f)}{r}=0,\end{eqnarray} then
\begin{eqnarray*}
m\left(r, \frac{f(z+c)}{f(z)}\right)+m\left(r, \frac{f(z)}{f(z+c)}\right)=o\left(T(r, f)\right)
\end{eqnarray*} for all $r\not\in E,$ where $E$ is a set with zero upper density measure $E,$ i.e., $$\overline{dens}E=\limsup_{r\rightarrow\infty}\frac{1}{r}\int_{E\cap [1,r]}dt=0.$$
\end{theorem}

\begin{remark}\label{R0}(i). We note that the condition \eqref{E1} implies that $\rho_{2}(f)\leq 1$ and the equality can possibly take happened. In fact, assume that \eqref{E1} holds, then there exists $r_{0}>0$ such that for any $r>r_{0},$ we have $\log T(r, f)<r$ and thus $\rho_{2}(f)\leq 1.$ Moreover, whenever $f$ is taken to satisfy, for example $T(r, f)=\exp\{\frac{r}{(\log r)^{m}}\}$ where $m\geq 1,$  one can easily get both \eqref{E1} and $\rho_{2}(f)=1.$ Hence, Theorem \ref{T1} is an improvement of  all the difference version of the logarithmic derivative lemma in several varaibles obtained before.\par

(ii). By the new version of the logarithmic difference lemma, all the second main theorem and Picard type theorem for meromorphic mappings from $\mathbb{C}^{m}$ into complex projective spaces $\mathbb{P}^{n}(\mathbb{C})$ obtained in \cite{korhonen-1, cao-1, cao-korhonen-2016} (including also \cite{halburd-korhonen-tohge-1, wond-law-wong-1, korhonen-li-tohge,  cao-nie}) can be improved under the assumption of \eqref{E1}.
\end{remark}

Before giving the proof, we show the following lemma proved recently by Zheng and Korhonen,  by which they obtained an improvement of difference version of logarithmic derivative lemma for meromorphic functions of one variable under the assumption \eqref{E1}.  This lemma is an improvement of a result on growth properties of nondecreasing continuous real functions (\cite[Lemma 2.1]{halburd-korhonen-1} and \cite{halburd-korhonen-tohge-1}).  Here the properties of real logarithmic convex functions are considered. Note that the characteristic function $T(r, f)$ and counting function $N(r,f)$ for a meromorphic function on $\mathbb{C}^{n}$ are satisfies the properties of nondecreasing positive, logarithmic convex, continuous function for $r.$ \par

\begin{lemma}\cite[Lemma 2.1]{zheng-korhonen-2018} \label{L1} Let $T(r)$ be a nondecreasing positive function in $[1, +\infty)$ and logarithmic convex with $T(r)\rightarrow+\infty (r\rightarrow+\infty).$  Assume that
\begin{equation}\label{E0}
\liminf_{r\rightarrow\infty} \frac{\log T(r)}{r}=0.
\end{equation} Set $$\phi(r)=\max_{1\leq t\leq r}\{\frac{t}{\log T(t)}\}.$$ Then given a constant $\delta\in(0, \frac{1}{2}),$ we have
\begin{equation*} T(r)\leq T(r+\phi^{\delta}(r))\leq \left(1+4\phi^{\delta-\frac{1}{2}}(r)\right)T(r), \,\, r\not\in E_{\delta},
\end{equation*}
where $E_{\delta}$ is a subset of $[1, +\infty)$ with the zero lower density. And $E_{\delta}$ has the zero upper density if \eqref{E0} holds for $\limsup.$
\end{lemma}

\begin{remark}Note that $\phi^{\delta}(r)\rightarrow\infty$ and $\phi^{\delta-\frac{1}{2}}(r)\rightarrow 0$ as $r\rightarrow \infty$ in Lemma \ref{L1}. Then for sufficiently large $r,$  we have $\phi^{\delta}(r)\geq h$ for any positive constant $h.$ Hence, $$T(r)\leq T(r+h)\leq T(r+\phi^{\delta}(r)) \leq (1+\varepsilon)T(r), \,\, r\not\in E,$$ where $E$ is a subset of $[1, +\infty)$ with the zero lower density.\end{remark} \par

The following lemma was obtained by Korhonen \cite{korhonen-1}. Since the assumption of $f(0)\neq 0, \infty$ for a meromorphic function $f$ of one variable in \cite[Lemma 5.1]{korhonen-1} can be omitted when the Poisson-Jensen formula is used, it does not matter with \cite[Lemma 5.2]{korhonen-1}. Thus we delete it in the statement.   \par

\begin{lemma}\cite[Lemma 5.2]{korhonen-1}\label{L3} Let $f$ be a nonconstant meromorphic function in $\mathbb{C}^{n},$ let $c=(c_{1}, \ldots, c_{n})\in\mathbb{C}^{n},$ let $\frac{1}{4}<\delta<1,$ and denote $\tilde{c}_{j}=(0,\ldots, 0, c_{j}, 0,\ldots, 0).$ Then there exists a nonnegative constant $C(\delta),$ depending only on $\delta,$ such that \begin{eqnarray*}
&&\int_{\partial B_{n}(r)}\log^{+}\left|\frac{f(z+\tilde{c}_{j})}{f(z)}\right|\sigma_{n}(z)\\&\leq& \frac{8\pi|c_{j}|^{\delta} C(\delta)}{\delta(1-\delta)}\left(\frac{R}{r}\right)^{2n-2}\frac{n_{f}(R, \infty)+n_{f}(R, 0)}{r^{\delta}}\\&&+\frac{4\pi|c_{j}|}{1-\delta}\left(\frac{R}{r}\right)^{2n-2}\left(\frac{R}{R-(r+|c_{j}|)}\right)\left(\frac{R}{R-r}\right)^{1-\delta}\frac{m_{f}(r, \infty)+m_{f}(r,0)}{\sqrt{R^{2}-r^{2}}}
\end{eqnarray*}for all $R>r+|c_{j}|>|c_{j}|.$
\end{lemma}

Now we give the proof of our version of logarithmic difference lemma.\par

\begin{proof}[Proof of Theorem \ref{T1}.] By the definition of counting function, we have $$n_{f}(r, \infty)+n_{f}(r,0)\leq\frac{R}{R-r}\left(N(R,f)+N(R, \frac{1}{f})\right)$$
for all $R>r.$ Then it follows by Lemma \ref{L3} and the first main theorem that there exists a positive constant $K_{1},$ depending only on $c_{j}=(0,\ldots, 0, c_{j}, 0,\ldots, 0)$ and $\delta^{'}\in(\frac{1}{4}, 1),$ such that \begin{eqnarray}\label{E2}m(r, \frac{f(z+\tilde{c}_{j})}{f(z)})&=&\int_{\partial B_{n}(r)}\log^{+}\left|\frac{f(z+\tilde{c}_{j})}{f(z)}\right|\sigma_{n}(z)\\\nonumber
&\leq&K_{1}K_{2}(r, R)\left(T(R, f)+\log\frac{1}{|f(0)|}\right)
\end{eqnarray} for all $R>r+|c_{j}|>|c_{j}|,$ where $$K_{2}(r, R)=\left(\frac{R}{r}\right)^{2n-2}\left(\frac{1}{R-(r+|c_{j}|)}\right)\left(\frac{R}{\sqrt{R^{2}-r^{2}}}
\left(\frac{R}{R-r}\right)^{1-\delta^{'}}+\frac{1}{r^{\delta^{'}}}\right).$$\par

Under the assumption of \eqref{E1}. Take $R=(r+|c_{j}|)+\frac{(r+|c_{j}|)^{\delta}}{(\log T(r+|c_{j}|, f))^{\delta}}, \delta\in(0, \frac{1}{2}).$ Then for sufficiently large $r,$
\begin{eqnarray*}
\frac{1}{R-(r+|c_{j}|)}&=&\left(\frac{\log T(r+|c_{j}|, f)}{r+|c_{j}|}\right)^{\delta}=o(1),
\end{eqnarray*}
\begin{eqnarray*}
\frac{R}{r}&=&1+\frac{|c_{j}|}{r}+\frac{(r+|c_{j}|)^{\delta}}{r})\frac{1}{(\log T(r+|c_{j}|, f))^{\delta}}
=o(1)
\end{eqnarray*}
and
\begin{eqnarray*}
\frac{R}{\sqrt{R^{2}-r^{2}}}\left(\frac{R}{R-r}\right)^{1-\delta^{'}}
= \frac{\frac{R}{r}}{\sqrt{\frac{R}{r}-1}}\left(\frac{\frac{R}{r}}{\frac{R}{r}-1}\right)^{1-\delta^{'}}
=o(1).
\end{eqnarray*} Combining these with \eqref{E2},   \begin{eqnarray}\label{E6}
m(r, \frac{f(z+\tilde{c}_{j})}{f(z)})&\leq& o(1)\left(T(R, f)+\log\frac{1}{|f(0)|}\right)
\end{eqnarray} for all sufficiently large $r.$ Moreover, under the assumption \eqref{E1}, it follows from Lemma \ref{L1} that for any $\varepsilon^{'}>0$ and $\phi(r)=\frac{r+|c_{j}|}{\log T(r+|c_{j}|, f)},$  \begin{eqnarray*}
T(R, f)\leq (1+\varepsilon^{'}(r)) T(r+|c_{j}|, f)\leq (1+\varepsilon^{'}(r))^{2} T(r, f)
\end{eqnarray*}  holds for all $r\not\in E_{1}$ where $\overline{dens}E_{1}=0.$ Hence, \eqref{E6} yields
\begin{eqnarray}\label{E8}
m(r, \frac{f(z+\tilde{c}_{j})}{f(z)})=\int_{\partial B_{n}(r)}\log^{+}\left|\frac{f(z+\tilde{c}_{j})}{f(z)}\right|\sigma_{n}(z)= o\left(T(r, f)\right)
\end{eqnarray} for all $r$ possibly outside  the set $E_{1}$ with $\overline{dens}E_{1}=0.$ \par

Now for any $c\in\mathbb{C}^{n},$ it can be written as $c=\tilde{c}_{1}+\cdots +\tilde{c}_{n}.$ Take $\tilde{c}_{0}=0.$ Since
\begin{eqnarray*}
&&\frac{f(z+c)}{f(z)}\\&=&\frac{f(z+(c_{1}, \ldots, c_{n}))}{f(z+(c_{1},\ldots, c_{n-1}, 0))}
\cdot\frac{f(z+(c_{1}, \ldots, c_{n-1}, 0))}{f(z+(c_{1}, \ldots, c_{n-2}, 0, 0))}\cdots\frac{f(z+(c_{1}, 0,\ldots,0))}{f(z+(0, \ldots, 0))}\\
&=&\frac{f(z+\sum_{j=0}^{n}\tilde{c}_{j})}{f(z+\sum_{j=0}^{n-1}\tilde{c}_{j})}\cdot\frac{f(z+\sum_{j=0}^{n-1}\tilde{c}_{j})}{f(z+\sum_{j=0}^{n-2}\tilde{c}_{j})}
\cdots\frac{f(z+\sum_{j=0}^{1}\tilde{c}_{j}))}{f(z+\tilde{c}_{0})},
\end{eqnarray*}
we get from \eqref{E8} that
\begin{eqnarray}\label{E9}
m(r, \frac{f(z+c)}{f(z)})&=&\sum_{j=1}^{n}o(T(r, f(z+\sum_{k=0}^{j-1}\tilde{c}_{k})))
\end{eqnarray}for all $r$ possibly outside the set $E_{1}$ with  $\overline{dens}E_{1}=0.$\par

Next, we assert that \begin{equation}\label{E10}T(r, f(z+c))=T(r, f)+o(T(r, f))\end{equation} for any $c=(c_{1}, c_{2}, \ldots, c_{n})$ and for all $r$ possibly outside a set $F$ with  $\overline{dens}(F)=0.$ In fact, by the fist main theorem and \eqref{E1}, we have $$\limsup_{r\rightarrow\infty}\frac{\log N(r,f) }{r}\leq\limsup_{r\rightarrow\infty}\frac{\log T(r, f) }{r}=0.$$ Then by Lemma \ref{L1} we get that
\begin{equation}\label{E11} N(r+h, f)=(1+o(1))N(r, f)\end{equation} holds for any constant $h(>0)$ independently on $r$ and all $r\not\in E_{2}$ with $\overline{dens}E_{2}=0.$ Hence, it follows from \eqref{E8} and \eqref{E11} that
\begin{eqnarray*}
T(r, f(z+\tilde{c}_{j})&=&m(r, f(z+\tilde{c}_{j}))+N(r, f(z+\tilde{c}_{j}))\\\nonumber
&\leq& m(r, \frac{f(z+\tilde{c}_{j})}{f(z)})+m(r,f)+N(r+|\tilde{c}_{j}|, f)\\\nonumber
&=& m(r, \frac{f(z+\tilde{c}_{j})}{f(z)})+m(r,f)+N(r, f)+o(N(r,f))\\\nonumber
&=&T(r, f)+o(T(r, f))
\end{eqnarray*}for all $r$ possibly outside the set $E_{1}\cup E_{2}$ with  $\overline{dens}(E_{1}\cup E_{2})=0.$ Thus, it deduces that
\begin{eqnarray*}
T(r, f(z+c))&\leq& T(r, f(z+(c_{1}, \ldots, c_{n-1}, 0)))+o(T(r, f(z+(c_{1}, \ldots, c_{n-1}, 0))))\\
&\leq& T(r, f(z+(c_{1}, \ldots, c_{n-1}, 0)))+o(T(r, f(z+(c_{1}, \ldots, c_{n-2}, 0, 0)))\\
&\vdots&\\
&\leq& T(r, f(z+(c_{1}, 0,\ldots, 0)))+o(T(r,f))\\
&\leq& T(r, f)+o(T(r, f))
\end{eqnarray*} for all $r$ possibly outside the set $F=E_{1}\cup E_{2}$ with  $\overline{dens}F=0.$ Note that $f(z)=f((z+c)-c)).$ Then we get the assertion.\par

Therefore, the theorem is got immediately from \eqref{E9} and \eqref{E10}.
\end{proof}

From the proof of Theorem \ref{T1}, we have the assertion \eqref{E10}. Since the relation between $T(r, f(z))$ and $T(r, f(z+c))$ is very useful to study solutions of complex difference equations, we here rewrite it as a theorem.\par

\begin{theorem}\label{T5} Let $f$ be a nonconstant meromorphic function on $\mathbb{C}^{n}$ with \begin{eqnarray*}
\limsup_{r\rightarrow\infty}\frac{\log T(r, f)}{r}=0,\end{eqnarray*} then
\begin{eqnarray*}
T(r, f(z+c))=T(r, f)+o(T(r, f))
\end{eqnarray*} holds for any constant $c\in\mathbb{C}^{n}\setminus\{0\}$ and all $r\not\in E$ with $\overline{dens}E=0.$
\end{theorem}

If using the Hinkkanen's Borel type Growth Lemma but not Lemma \ref{L1}, we can obtain another form of the logarithmic difference lemma  as follows. A tropical version is also given by Cao and Zheng \cite{cao-zheng-2018} at the same time.\par

\begin{theorem}\label{T1''}
Let $f$ be a nonconstant meromorphic function on $\mathbb{C}^{n},$ and let $c\in\mathbb{C}^{n}\setminus\{0\}.$ If \begin{eqnarray}\label{E1'}\limsup_{r\rightarrow\infty}\frac{\log T(r, f)(\log r)^{\varepsilon}}{r}=0,\end{eqnarray} for any $\varepsilon(>0),$ then
\begin{eqnarray*}
m\left(r, \frac{f(z+c)}{f(z)}\right)+m\left(r, \frac{f(z)}{f(z+c)}\right)=o\left(T(r, f)\right)
\end{eqnarray*} for all $r\not\in E,$ where $E$ is a set with $\int_{E}\frac{dt}{t\log t}<+\infty$ which implies $E$  with zero upper logarithmic density measure i.e., $$\overline{dens}E=\limsup_{r\rightarrow\infty}\frac{1}{\log r}\int_{E\cap [1,r]}\frac{dt}{t}=0.$$
\end{theorem}

The next lemma is the Hinkkanen's Borel type growth lemma (or see also a similar lemma
\cite[Lemma 3.3.1]{cherry-ye}.\par

\begin{lemma}\cite[Lemma 4]{hinkkanen} \label{L2} Let $p(r)$ and $h(r)=\varphi(r)/r$ be positive nondecreasing functions defined for $r\geq\varrho>0$ and $r\geq\tau>0,$ respectively, such that $\int_{\varrho}^{\infty}\frac{dr}{p(r)}=\infty$ and $\int_{\tau}^{\infty}\frac{dr}{\varphi(r)}<\infty.$ Let $u(r)$ be a positive nondecreasing function defined for $r\geq r_{0}\geq\varrho$ such that $u(r)\rightarrow\infty$ as $r\rightarrow\infty.$ Then if $C$ is real with $C>1,$ we have $$u(r+\frac{p(r)}{h(u(r))})<Cu(r)$$ whenever $r\geq r_{0},$ $u(r)>\tau,$ and $r\not\in E$ where $$\int_{E}\frac{dr}{p(r)}\leq\frac{1}{h(w)}+\frac{C}{C-1}\int_{w}^{\infty}\frac{dr}{\varphi(r)}<\infty$$ and $w=\max\{\tau, u(r_{0})\}.$
\end{lemma}

\begin{proof}[Proof of Theorem \ref{T1''}.]  In Lemma \ref{L2}, we take $$u(r)=T(r, f),\quad p(r)=r\log r,$$ and $$h(r)=\frac{\varphi(r)}{r}$$ where $\varphi(r)=r\log r (\log\log  r)^{1+\varepsilon}$ with $\varepsilon>0.$ Then it is obvious that $\int_{\varrho}^{\infty}\frac{dr}{p(r)}=\infty$ and $\int_{\tau}^{\infty}\frac{dr}{\varphi(r)}<\infty$ for $r\geq\varrho>0$ and $r\geq\tau>0.$ Let \begin{eqnarray*}
R&:=&(r+|c_{j}|)+\frac{p(r+|c_{j}|)}{(r+|c_{j}|)h(T_{f}(r+|c_{j}|))}\\&=&(r+|c_{j}|)+\frac{(r+|c_{j}|)\log (r+|c_{j}|)}{\log T_{f}(r+|c_{j}|)(\log\log T_{f}(r+|c_{j}|))^{1+\varepsilon}}.
\end{eqnarray*}

Note that $$T(R, f)=T\left((r+|c_{j}|)+\frac{(r+|c_{j}|)\log (r+|c_{j}|)}{\log T_{f}(r+|c_{j}|)(\log\log T_{f}(r+|c_{j}|))^{1+\varepsilon}}, f\right).$$ Applying Lemma \ref{L2}, we have
\begin{equation}\label{E3}
T(R, f)\leq C T(r+|c_{j}|, f)
\end{equation} for all $r$ possibly outside a set $E_{1}$ satisfying
\begin{equation*}
E_{1}:=\{r\in[r_{0}, \infty): T(R, f)\geq C T(r+|c_{j}|, f) \}
\end{equation*}where \begin{eqnarray*} \int_{E_{1}}\frac{dt}{p(t)}&=&\int_{E_{1}}\frac{dt}{t\log t}\\&\leq&\frac{1}{\log w(\log \log w)^{1+\varepsilon}}+\frac{C}{C-1}\int_{w}^{\infty}\frac{dt}{t \log t(\log\log t)^{1+\varepsilon}}\\&<&+\infty.\end{eqnarray*} This gives \begin{eqnarray*} \overline{logdens}E_{1}&=&\limsup_{r\rightarrow\infty}\frac{1}{\log r}\int_{E_{1}\cap[1,r]}\frac{dt}{t}\\&\leq&\limsup_{r\rightarrow\infty}\frac{\int_{E_{1}\cap[1, \log r]}\frac{dt}{t}}{\log r}+\limsup_{r\rightarrow\infty}\int_{E_{1}\cap[\log r, r]}\frac{dt}{t\log t}\\&\leq&\limsup_{r\rightarrow\infty}\frac{\log\log r}{\log r}+0=0.\end{eqnarray*}

Under the condition \eqref{E1'}, we get that for any $\varepsilon^{'}>0$ and sufficiently large $r,$
\begin{eqnarray}\label{E4}
\frac{\log T(r+|c_{j}|, f)(\log (r+|c_{j}|))^{\varepsilon}}{r+|c_{j}|}<\varepsilon^{'}.
\end{eqnarray} Since \eqref{E1'} implies $\rho_{2}(f)\leq 1$ according to Remark \ref{R0}(i), we have
\begin{eqnarray}\label{E5}
\frac{\log\log T(r+|c_{j}|, f)}{\log (r+|c_{j}|)}\leq 1+\varepsilon^{''}\end{eqnarray}
for any $\varepsilon^{''}>0$ and sufficiently large $r.$ Then \eqref{E4} and \eqref{E5} give that for sufficiently large $r,$
\begin{eqnarray*}
\frac{1}{R-(r+|c_{j}|)}&=&\frac{\log T(r+|c_{j}|, f)(\log\log T(r+|c_{j}|, f))^{1+\varepsilon}}{(r+|c_{j}|)\log (r+|c_{j}|)}\\
&=&\frac{\log T(r+|c_{j}|, f) (\log (r+|c_{j}|))^{\varepsilon}}{r+|c_{j}|}\left(\frac{\log\log T(r+|c_{j}|, f)}{\log (r+|c_{j}|)}\right)^{1+\varepsilon}\\
&\leq&\varepsilon^{'} (1+\varepsilon^{''})^{1+\varepsilon},
\end{eqnarray*}

\begin{eqnarray*}
\frac{R}{r}&=&1+\frac{|c_{j}|}{r}+(1+\frac{|c_{j}|}{r})\frac{\log (r+|c_{j}|)}{\log T(r+|c_{j}|, f)(\log\log T(r+|c_{j}|, f))^{1+\varepsilon}}\\
&=&1+\frac{|c_{j}|}{r}+(1+\frac{|c_{j}|}{r})\frac{1}{(\log (r+|c_{j}|))^{\varepsilon}\log T(r+|c_{j}|, f)\left(\frac{\log\log T(r+|c_{j}|, f)}{\log (r+|c_{j}|)}\right)^{1+\varepsilon}}\\
&=&o(1)
\end{eqnarray*}
and
\begin{eqnarray*}
\frac{R}{\sqrt{R^{2}-r^{2}}}\left(\frac{R}{R-r}\right)^{1-\delta^{'}}
= \frac{\frac{R}{r}}{\sqrt{\frac{R}{r}-1}}\left(\frac{\frac{R}{r}}{\frac{R}{r}-1}\right)^{1-\delta^{'}}
=o(1).
\end{eqnarray*}
Combining these with \eqref{E2} and \eqref{E3},   \begin{eqnarray}\label{E6'}
m(r, \frac{f(z+\tilde{c}_{j})}{f(z)})&\leq& o(1)\left(T(r+|c_{j}|, f)+\log\frac{1}{|f(0)|}\right)
\end{eqnarray} for all $r$ possibly outside a set $E_{1}$ with $\int_{E_{1}}\frac{dt}{t\log t}<+\infty.$\par

By \eqref{E3}, we also have $$T\left(r+\frac{r\log r}{\log T(r, f)(\log\log T(r, f))^{1+\varepsilon}}, f\right)\leq C T(r, f)$$ for all $r\not\in E_{1}.$ It follows from \eqref{E4} and \eqref{E5} that
\begin{eqnarray*}
\frac{\log T(r, f)(\log r)^{\varepsilon}}{r}<\varepsilon^{'}
\end{eqnarray*} and
\begin{eqnarray*}
\frac{\log\log T(r, f)}{\log r)}\leq 1+\varepsilon^{''}.\end{eqnarray*} Thus it yields that
$$ \frac{r\log r}{\log T(r, f)(\log\log T(r, f))^{1+\varepsilon}}\rightarrow\infty$$
as $r\rightarrow\infty.$ Then we have $$r+|c_{j}|\leq r+\frac{r\log r}{\log T(r, f)(\log\log T(r, f))^{1+\varepsilon}}$$ for sufficiently large $r.$ Hence, \begin{eqnarray}\label{E7}T(r+|c_{j}|, f)\leq T(r+\frac{r\log r}{\log T(r, f)(\log\log T(r, f))^{1+\varepsilon}}, f)\leq C T(r ,f)\end{eqnarray} for all $r\not\in E_{1}.$ Therefore,  we get from \eqref{E6'} and \eqref{E7} that the equation \eqref{E8} is still valid for $r$ possibly outside  the set $E_{1}.$ Using as the same reason as in the proof of Theorem \ref{T1} to get \eqref{E9}, we then get immediately the conclusion of the theorem from \eqref{E9} and Theorem \ref{T5}.
\end{proof}

In the proof of Theorem \ref{T1''}, we do not know how to improve  the condition \eqref{E1'} by \eqref{E1}  whenever using the Hinkkanen's Borel type Growth Lemma (Lemma \ref{L2}). The difficulty we met is how to give well defined functions $p(r)$ and $\varphi(r)$ when applying Lemma \ref{L2}. After finished this paper, we learn that Korhonen- Tohge-Zhang-Zheng \cite[Lemma 3.1]{korhonen-tohge-zhang-zheng} recently obtained a similar result on the logarithmic  difference lemma for meromorphic functions in one variable under the assumption of
\begin{eqnarray}\label{E-1.5}\log T(r,f)\leq \frac{r}{(\log r)^{2+\nu}}\end{eqnarray}
for any $\nu(>0).$ It is easy to see that this assumption \eqref{E-1.5} is stronger than \eqref{E1'}. Hence, Theorem \ref{T1''} (and thus Theorem \ref{T1}) is an improvement and extension of their result. \par

For study on the solutions of complex partial difference equations, we next prove another form of the logarithmic difference lemma for meromorphic functions with finite order in several complex variables. This is an extension of \cite[Corollary 2.5]{chiang-feng} from one variable to several variables.\par

\begin{theorem}\label{T1'}
Let $f$ be a nonconstant meromorphic function on $\mathbb{C}^{n}$ and let $c\in\mathbb{C}^{n}\setminus\{0\}.$ If $f$ is of finite order, then
\begin{eqnarray*}
m\left(r, \frac{f(z+c)}{f(z)}\right)+m\left(r, \frac{f(z)}{f(z+c)}\right)=O\left(r^{\rho(f)-1+\varepsilon}\right)
\end{eqnarray*} holds for any $\varepsilon(>0).$
\end{theorem}

\begin{proof} Since $f$ is of finite order, $T(r, f)\leq r^{\rho(f)+\varepsilon}$ holds for any $\varepsilon>0.$ Take $R=2r.$ Then it follows from \eqref{E2} that
\begin{eqnarray*}m(r, \frac{f(z+\tilde{c}_{j})}{f(z)})=O(r^{\rho(f)-1+\varepsilon}).\end{eqnarray*}
For any $c\in\mathbb{C}^{n}$ which can be written as $c=\tilde{c}_{1}+\cdots +\tilde{c}_{n}.$ Take $\tilde{c}_{0}=0.$ Since
\begin{eqnarray*}
&&\frac{f(z+c)}{f(z)}\\&=&\frac{f(z+(c_{1}, \ldots, c_{n}))}{f(z+(c_{1},\ldots, c_{n-1}, 0))}
\cdot\frac{f(z+(c_{1}, \ldots, c_{n-1}, 0))}{f(z+(c_{1}, \ldots, c_{n-2}, 0, 0))}\cdots\frac{f(z+(c_{1}, 0,\ldots,0))}{f(z+(0, \ldots, 0))}\\
&=&\frac{f(z+\sum_{j=0}^{n}\tilde{c}_{j})}{f(z+\sum_{j=0}^{n-1}\tilde{c}_{j})}\cdot\frac{f(z+\sum_{j=0}^{n-1}\tilde{c}_{j})}{f(z+\sum_{j=0}^{n-2}\tilde{c}_{j})}
\cdots\frac{f(z+\sum_{j=0}^{1}\tilde{c}_{j}))}{f(z+\tilde{c}_{0})},
\end{eqnarray*}
we then get that \begin{eqnarray}&&m(r, \frac{f(z+c)}{f(z)})\\\nonumber
&=&O\left(r^{\rho(f)-1+\varepsilon}+r^{\rho(f(z+\sum_{j=0}^{1}\tilde{c}_{j}))-1+\varepsilon}+\ldots+r^{\rho((f(z+\sum_{j=0}^{n-1}\tilde{c}_{j}))-1+\varepsilon}\right).\end{eqnarray}
The assumption $\rho(f)<\infty$ implies that we can get from Theorem \ref{T5} that $\rho(f)=\rho(f(z+\sum_{j=0}^{1}\tilde{c}_{j}))=\ldots=\rho(f((z+\sum_{j=0}^{n-1}\tilde{c}_{j})))=\rho(f(z+c)).$ Therefore, the conclusion of this theorem is true.\end{proof}

By Lemma \ref{L1}, one can get that $N(r+|c|,f)=N(r, f)+o(N(r,f))$ for $r\not\in E$ with $\overline{dens}E=0$ under the assumption of $\limsup_{r\rightarrow\infty}\frac{\log N(r,f)}{r}=0.$ Note that $N(r, f(z+c))\leq N(r+|c|, f)$ by the definition of counting function. Hence, provided that $\limsup_{r\rightarrow\infty}\frac{\log N(r,f)}{r}=0,$  we have \begin{equation}\label{E9.0}N(r, f(z+c))=N(r, f)+o(N(r,f))\end{equation} for $r\not\in E.$  Below, we get a more  explicit relationship between $N(r,f(z+c))$ and  $N(r,f)$ for finite convergence exponent of poles (and thus true also for finite order). This is an extension of \cite[Theorem 2.2]{chiang-feng} from one variable to several variables.\par

\begin{theorem}\label{T6} Let the convergence exponent of poles of a meromorphic function $f$ on $\mathbb{C}^{n}$ be finite, i.e.,
\begin{eqnarray*}\lambda(\frac{1}{f}):=\limsup_{r\rightarrow\infty}\frac{\log N(r, f)}{\log r}<\infty,\end{eqnarray*} then  for any $c\in\mathbb{C}^{n}\setminus\{0\},$
\begin{equation*} N(r, f(z+c))=N(r, f)+O(r^{\lambda(\frac{1}{f})-1+\varepsilon})\end{equation*} holds for any $\varepsilon>0.$ The $\lambda(\frac{1}{f})$ can be changed by $\rho(f)$ whenever $f$ is of finite order.
 \end{theorem}

\begin{proof}Set $\|c\|=\sqrt{|c_{1}|^{2}+\ldots+|c_{m}|^{2}}.$ Sine $\lambda(\frac{1}{f})<\infty,$ it is enough to take the same method due to Zheng and Korhonen \cite[Pages 15-16]{zheng-korhonen-2018} to obtain that \begin{equation*} N(r+\|c\|, f)=N(r, f)+O(r^{\lambda(\frac{1}{f})-1+\varepsilon})\end{equation*} holds for any $\varepsilon>0.$ The original proof is owing to Chiang and Feng \cite[Theorem 2.2]{chiang-feng} by the definition of Riemann-Stieltjes integral  for counting functions (in fact, they proved this lemma for meromorphic functions of one variables). On the other hand, it is obvious that $$N(r, f(z+c))\leq N(r+\|c\|, f)$$ by the definition of counting function. Hence, we get that \begin{equation*} N(r, f+c)\leq N(r, f)+O(r^{\lambda(\frac{1}{f})-1+\varepsilon})\end{equation*} and thus \begin{equation*} N(r, f)\leq N(r, f+c)+O(r^{\lambda(\frac{1}{f+c})-1+\varepsilon})\end{equation*} holds for any $\varepsilon>0.$ The assumption $\lambda(\frac{1}{f})<\infty$ implies that we can get from \eqref{E9.0} that $\lambda(\frac{1}{f})=\lambda(\frac{1}{f(z+c)}).$ Therefore,
\begin{equation*} N(r, f+c)=N(r, f+c)+O(r^{\lambda(\frac{1}{f})-1+\varepsilon})\end{equation*} holds for any $\varepsilon>0.$ Obviously, the $\lambda(\frac{1}{f})$ can be changed by the order of $f$ from the above discussion whenever $f$ is of finite order.
\end{proof}

Finally in this section, we give the  explicit relation $T(r,f(z+c))\sim T(r,f)$ for a meromorphic function with finite order. This is an extension of \cite[Theorem 2.1]{chiang-feng}.\par

\begin{theorem} \label{T10}If  a meromorphic function $f$ on $\mathbb{C}^{n}$ is of finite order, then  \begin{equation*}T(r, f(z+c))=T(r, f)+O(r^{\rho(f)-1+\varepsilon})\end{equation*} for any $c\in\mathbb{C}^{n}\setminus\{0\}$ and for any $\varepsilon>0.$\end{theorem}

\begin{proof}By Theorem \ref{T1'} and Theorem \ref{T6}, we have
\begin{eqnarray*}
T(r, f(z+c))&=&m(r, f(z+c))+N(r, f(z+c))\\\nonumber
&\leq& m(r, \frac{f(z+c)}{f(z)})+m(r,f)+N(r, f)+O(r^{\rho(f)-1+\varepsilon})\\\nonumber
&=&T(r, f)+O(r^{\rho(f)-1+\varepsilon}).
\end{eqnarray*}This implies
\begin{eqnarray*}
T(r, f)&\leq&T(r, f(z+c))+O(r^{\rho(f(z+c))-1+\varepsilon}).
\end{eqnarray*}Since $\rho(f)<\infty,$ it follows from Theorem \ref{T5} that $\rho(f(z+c))=\rho(f).$ Hence the theorem is proved.
\end{proof}

\section{Partial difference equations}\label{s3}

In this section, we will consider meromorphic solutions of partial difference equations by making use of our results on logarithmic difference lemma. Recall that a meromorphic function $g$ is said to be a small function with respect to another meromorphic function $f$ if $T(r, g)=o(T(r, f)).$ For examples, constant functions are small with respect to rational functions, and finite order meromorphic functions are small with respect to infinite order meromorphic functions.  A meromorphic solution $w$ on $\mathbb{C}^{n}$ of a partial difference equation (or even a general form of functional equation) is called admissible if all coefficients $\{a_{j}\}$ of the equation are small functions with respect to $w.$ \par

\subsection{Some nonlinear partial difference equations}\noindent

Let us start with the discrete KdV equation \cite{hirota} of the form $X^{i+1}_{j+1}=X_{j}^{i}+\frac{1}{X_{j+1}^{i}}-\frac{1}{X_{j}^{i+1}}.$ Since this form is not very convenient and thus its potential form $X^{i+1}_{j+1}=X_{j}^{i}+\frac{Z_{j}^{i}}{X_{j}^{i+1}-X_{j+1}^{i}}$ was studied (see \cite{tremblay-grammaticos-ramani}). Motivated by the discrete potential KdV equation,  we consider the partial difference equation as follows.\par

\begin{theorem}\label{T10'} Let $c_{1}, c_{2}\in\mathbb{C}\setminus\{0\}.$ Let $f$ be a nontrivial meromorphic solution of the partial difference equation
\begin{eqnarray}\label{E7.8}f(z_{1}+c_{1}, z_{2}+c_{2})=f(z_{1}, z_{2})+\frac{A(z_{1}, z_{2})}{f(z_{1}, z_{2}+c_{2})-f(z_{1}+c_{1}, z_{2})},\end{eqnarray} where $A(z_{1}, z_{2})$ is a nonzero meromorphic function on $\mathbb{C}^{2}$ small with respect to the solution $f,$ that is $T(r, A)=o(T(r,f)).$
If  $\delta_{f}(0)>0,$  then  $$\limsup_{r\rightarrow\infty}\frac{\log T(r, f)}{r}>0.$$\end{theorem}

\begin{proof} Assume that a nontrivial meromorphic solution $f$ satisfies the condition of $\limsup_{r\rightarrow\infty}\frac{\log T(r, f)}{r}=0.$ Since $A$ is a nonzero meromorphic function small with respect to $f,$ we get from the first main theorem that $$m(r, \frac{1}{A})\leq T(r, \frac{1}{A})=T(r, A)+O(1)=o(T(r, f)).$$ It follows from the equation \eqref{E7.8} that
\begin{eqnarray*}\frac{1}{f^{2}(z_{1}, z_{2})}=\frac{1}{A(z_{1}, z_{2})}\left(\frac{f(z_{1}+c_{1}, z_{2}+c_{2})}{f(z_{1}, z_{2})}-1\right)
\left(\frac{f(z_{1}, z_{2}+c_{2})}{f(z_{1}, z_{2})}-\frac{f(z_{1}+c_{1}, z_{2})}{f(z_{1}, z_{2})}\right),\end{eqnarray*}
 and thus, \begin{eqnarray*}m(r, \frac{1}{f^{2}(z_{1}, z_{2})})&\leq& m\left(r,\frac{f(z_{1}+c_{1}, z_{2}+c_{2})}{f(z_{1}, z_{2})}\right)+m\left(r, \frac{f(z_{1}, z_{2}+c_{2})}{f(z_{1}, z_{2})}\right)\\&&+m\left(r, \frac{f(z_{1}+c_{1}, z_{2})}{f(z_{1}, z_{2})}\right)+m(r, \frac{1}{A(z_{1}, z_{2})})+O(1)\\
 &\leq& m\left(r,\frac{f(z_{1}+c_{1}, z_{2}+c_{2})}{f(z_{1}, z_{2})}\right)+m\left(r, \frac{f(z_{1}, z_{2}+c_{2})}{f(z_{1}, z_{2})}\right)\\&&+m\left(r, \frac{f(z_{1}+c_{1}, z_{2})}{f(z_{1}, z_{2})}\right)+o(T(r,f)).\end{eqnarray*}
By Theorem \ref{T1}, one can deduce that  \begin{eqnarray*}&&m\left(r,\frac{f(z_{1}+c_{1}, z_{2}+c_{2})}{f(z_{1}, z_{2})}\right)+m\left(r, \frac{f(z_{1}+c_{1}, z_{2})}{f(z_{1}, z_{2})}\right)+m\left(r, \frac{f(z_{1}, z_{2}+c_{2})}{f(z_{1}, z_{2})}\right)\\&=&o(T(r, f))
\end{eqnarray*}hold for $r\not\in E$ where $E$ is a set with $\overline{dens}E=0.$ Hence,
\begin{eqnarray*} T(r, f^{2})&=&T(r, \frac{1}{f^{2}})+O(1)=N(r, \frac{1}{f^{2}})+m(r, \frac{1}{f^{2}})+O(1)\\
&\leq& 2N(r, \frac{1}{f})+o(T(r, f)) \end{eqnarray*}  holds for all $r\not\in E$ where $E$ is a set with $\overline{dens}E=0.$  Since $\delta_{f}(0)>0,$ we have
$$N(r, \frac{1}{f})<(1-\frac{\delta_{f}(0)}{2})T(r, f).$$ This gives
\begin{eqnarray*} T(r, f^{2})\leq 2(1-\frac{\delta_{f}(0)}{2})T(r, f)+o(T(r, f)) \end{eqnarray*}  holds for all $r\not\in E$ where $E$ is a set with $\overline{dens}E=0.$ By the Valion-Mohon'ko theorem in several complex variables \cite[Theorem 3.4]{hu-yang-1}, we get $$T(r, f^{2})=2 T(r, f)+o(T(r, f)).$$ Therefore, it follows that
$$\delta_{f}(0)T(r,f)\leq o(T(r,f))$$
for all $r\not\in E$ where $E$ is a set with $\overline{dens}E=0.$ This is a contradiction.
\end{proof}

\begin{example}
Let $c_{1}\in \mathbb{C}\setminus\{0\},$ $c_{2}=2\pi i.$ Then the transcendental meromorphic function $f(z_{1}, z_{2})=\frac{z_{1}^{2}}{z_{2}}+e^{z_{2}}$ is a solution of the partial difference equation \eqref{E7.8} with rational coefficient $A(z_{1}, z_{2})=\frac{-(z_{1}^{2}+2c_{1}z_{1})(2c_{1}z_{1}z_{2}+c_{1}^{2}z_{2}-2\pi iz_{1}^{2})}{z_{2}^{2}(z_{2}+2\pi i)}.$
One can deduce that $\delta_{f}(0)=0,$ $\rho(f)=1$ and $\limsup_{r\rightarrow\infty}\frac{\log T(r, f)}{r}=0.$ This means that the assumption $\delta_{f}(0)>0$ in Theorem \ref{T10'} is necessary.
\end{example}

\begin{example}
Denote $\wp(z_{2})$ by the Weierstrass $\wp$-function (an elliptic function) of one variable $z_{2}$ with two period $w_{1}$ and $w_{2}$ such that $\frac{w_{1}}{w_{2}}\not\in \mathbb{R}$ defined as $$\wp(z_{2})=\frac{1}{z_{2}}+\sum_{\mu,\nu; \mu^{2}+\nu^{2}\neq 0}\left\{\frac{1}{(z_{2}+\mu w_{1}+\nu w_{2})^{2}}-\frac{1}{(\mu w_{1}+\nu w_{2})^{2}}\right\},$$ which is even and satisfies the differential equation $(\wp(z_{2})^{'})^{2}=4\wp(z_{2})^{3}-1$ after appropriately choosing $w_{1}$ and $w_{2}.$ It was proved by Bank and Langley \cite[Corollary 2]{bank-langley} that $T(r, \wp)=O(r^{2})$ and $m(r, \wp)=o(r^{2})=o(T(r, \wp)).$ Then the meromorphic function $f(z_{1}, z_{2})=2 z_{1}+\wp(z_{2})$ is a solution of the partial difference equation $$f(z_{1}+w_{1}, z_{2}+w_{2})=f(z_{1}, z_{2})+\frac{4w_{1}^{2}}{f(z_{1}, z_{2}+w_{2})-f(z_{1}+w_{1}, z_{2})}.$$ Obviously, we have $\rho(f)=2$ and thus $\limsup_{r\rightarrow\infty}\frac{\log T(r, f)}{r}=0.$ It seems not easy to compute $N(r, \frac{1}{2 z_{1}+ \wp(z_{2})}),$ however, it is interesting that by Theorem \ref{T10'}, we get immediately $\delta_{f}(0)=0,$ and thus $N(r, \frac{1}{2 z_{1}+ \wp(z_{2})})=O(r^{2}).$ Of course, if one can compute directly $N(r, \frac{1}{2 z_{1}+ \wp(z_{2})})=O(r^{2}),$ then this implies that the condition $\delta_{f}(0)>0$ in Theorem \ref{T10'} is necessary.
\end{example}

\begin{question} Observe that any nonzero meromorphic function with two periods $(c_{1}, 0)$ and $(0, c_{2})$ must satisfy the discrete KdV partial difference equation
\begin{equation}\label{EQ}f(z_{1}+c_{1}, z_{2}+c_{2})=f(z_{1}, z_{2})+\frac{1}{f(z_{1}+c_{1}, z_{2})}-\frac{1}{f(z_{1}, z_{2}+c_{2})}.\end{equation} For example, $f(z_{1}, z_{2})=\sin z_{1}-\wp(z_{2})$ satisfies the equation \eqref{EQ} with $c_{1}=2k\pi$ and $c_{2}=w_{j}$ $(j\in\{1, 2\}).$ Thus it is interesting to ask whether all meromorphic solutions of the discrete KdV partial difference equation \eqref{EQ} must be period or not?
\end{question}

Recall that Gross \cite{gross} and Hayman \cite{hayman-Fermat} investigated meromorphic solutions of the Fermat functional equations $f^{m}+g^{m}=1$ and $f^{m}+g^{m}+h^{m}=1$ of one variable, respectively. Motivated by this, we get an interesting result on solutions of  nonlinear Fermat type partial difference equations as follows.

\begin{theorem}\label{TFE} Let $c_{1}, c_{2}\in\mathbb{C}\setminus\{0\}.$ Suppose that $f$ is a nontrivial meromorphic solution of the Fermat type partial difference equations
\begin{eqnarray}\label{E99}\frac{1}{f^{m}(z_{1}+c_{1}, z_{2}+c_{2})}+\frac{1}{f^{m}(z_{1}, z_{2})}=A(z_{1}, z_{2})f^{n}(z_{1}, z_{2}),\end{eqnarray} or
\begin{eqnarray}\label{E100}&&\\\nonumber&&\frac{1}{f^{m}(z_{1}+c_{1}, z_{2}+c_{2})}+\frac{1}{f^{m}(z_{1}+c_{1}, z_{2})}+\frac{1}{f^{m}(z_{1}, z_{2}+c_{2})}=A(z_{1}, z_{2})f^{n}(z_{1}, z_{2}),\end{eqnarray} where $m\in\mathbb{N},$  $n\in\mathbb{N}\cup\{0\},$ and $A(z_{1}, z_{2})$ is a nonzero meromorphic  function on $\mathbb{C}^{2}$ with respect to the solution $f,$ that is $T(r, A)=o(T(r, f)).$ If  $\delta_{f}(\infty)>0,$  then  $$\limsup_{r\rightarrow\infty}\frac{\log T(r, f)}{r}>0.$$ \end{theorem}

\begin{proof} Since $A$ is a small function with respect to $f,$ we get $$m(r, \frac{1}{A})\leq T(r, \frac{1}{A})=T(r, A)+O(1)=o(T(r, f)).$$
Since $\delta_{f}(\infty)>0,$ we have
$$N(r, f)<(1-\frac{\delta_{f}(\infty)}{2})T(r, f).$$
By the Valion-Mohon'ko theorem in several complex variables \cite[Theorem 3.4]{hu-yang-1}, we have $$T(r, f^{m+n})=(m+n) T(r, f)+o(T(r, f)).$$
Furthermore, by Theorem \ref{T1}, we have \begin{eqnarray*}&&m\left(r,\frac{f(z_{1}, z_{2})}{f(z_{1}+c_{1}, z_{2}+c_{2})}\right)+m\left(r, \frac{f(z_{1}, z_{2})}{f(z_{1}+c_{1}, z_{2})}\right)+m\left(r, \frac{f(z_{1}, z_{2})}{f(z_{1}, z_{2}+c_{2})}\right)\\&=&o(T(r, f))
\end{eqnarray*}hold for $r\not\in E$ where $E$ is a set with $\overline{dens}E=0.$\par

It follows from the Fermat type partial difference equation \eqref{E100} that
\begin{eqnarray*}&&f^{n+m}(z_{1}, z_{2})\\&=&\frac{1}{A(z_{1}, z_{2})}\left[\left(\frac{f(z_{1}, z_{2})}{f(z_{1}+c_{1}, z_{2}+c_{2})}\right)^{m}+
\left(\frac{f(z_{1}, z_{2})}{f(z_{1}, z_{2}+c_{2})}\right)^{m}+\left(\frac{f(z_{1}, z_{2})}{f(z_{1}+c_{1}, z_{2})}\right)^{m}\right].\end{eqnarray*} Therefore, we have
 \begin{eqnarray*}&&(m+n)T(r, f)\\&=&T(r, f^{m+n})+o(T(r, f))\\&=&N(r, f^{m+n})+m(r, f^{m+n})+o(T(r, f))\\&
 =&N(r, f^{m+n})+m(r, \frac{1}{A})+m\cdot m\left(r,\frac{f(z_{1}+c_{1}, z_{2}+c_{2})}{f(z_{1}, z_{2})}\right)\\
 &&+m\cdot m\left(r, \frac{f(z_{1}, z_{2}+c_{2})}{f(z_{1}, z_{2})}\right)+m\cdot m\left(r, \frac{f(z_{1}+c_{1}, z_{2})}{f(z_{1}, z_{2})}\right)+o(T(r, f))\\
 &\leq&(m+n)(1-\frac{\delta_{f}(\infty)}{2})T(r, f)+o(T(r,f)).\end{eqnarray*} We obtain a contradiction. Similarly discussion to the equation \eqref{E99} also gives a contradiction. \end{proof}

\begin{example} \cite[Example 1.5]{xu-cao} Let $c_{1}$ and $c_{2}$ be two complex values such that $c_{1}+2ic_{2}=-\frac{\pi}{2}+2k\pi$ $(k\in\mathbb{Z}).$ Then the meromorphic function $f(z_{1}, z_{2})=\frac{1}{\cos(z_{1}, 2i z_{2})}$ is a solution of the Fermat type partial difference equation $\frac{1}{f^{2}(z_{1}+c_{1}, z_{2}+c_{2})}+\frac{1}{f^{2}(z_{1}, z_{2})}=1.$ Obviously,  $\delta_{f}(\infty)=0,$ $\rho(f)=1$ and $\limsup_{r\rightarrow\infty}\frac{\log T(r, f)}{r}=0.$ This shows that the assumption  $\delta_{f}(\infty)>0$ in Theorem \ref{TFE} is necessary. \end{example}

Since $\delta_{f}(\infty)>0$ is always true for any nonconstant entire function, Theorem \ref{TFE} tells us that the equations \eqref{E99} and \eqref{E100} do not have any admissible entire solutions of $\limsup_{r\rightarrow\infty}\frac{\log T(r, f)}{r}=0.$ Furthermore, if take $g(z_{1}, z_{2})=\frac{1}{f(z_{1}, z_{2})},$ then a corollary is obtained immediately by Theorem \ref{TFE}.\par

\begin{corollary} Let $c_{1}, c_{2}\in\mathbb{C}\setminus\{0\}.$ Suppose that $g$ is a nontrivial meromorphic solution of the Fermat type partial difference equations
\begin{eqnarray*}g^{m}(z_{1}+c_{1}, z_{2}+c_{2})+g^{m}(z_{1}, z_{2})=\frac{A(z_{1}, z_{2})}{g^{n}(z_{1}, z_{2})},\end{eqnarray*} or
\begin{eqnarray*}\nonumber&&g^{m}(z_{1}+c_{1}, z_{2}+c_{2})+g^{m}(z_{1}+c_{1}, z_{2})+g^{m}(z_{1}, z_{2}+c_{2})=\frac{A(z_{1}, z_{2})}{g^{n}(z_{1}, z_{2})},\end{eqnarray*} where $m\in\mathbb{N},$  $n\in\mathbb{N}\cup\{0\},$ and $A(z_{1}, z_{2})$ is a nonzero meromorphic function on $\mathbb{C}^{2}$ with respect to the solution $g,$ that is $T(r, A)=o(T(r, g)).$ If  $\delta_{g}(0)>0,$  then  $\limsup_{r\rightarrow\infty}\frac{\log T(r, g)}{r}>0.$ \end{corollary}

\subsection{Linear partial difference equations}\noindent

Next, owing to many models of partial difference equations such as discrete heat equation, discrete Laplace equation, nonsymmetric partial difference functional equation, and discrete Poisson equation and others, we consider general partial linear difference equations, and obtain the following results. The first theorem extend and generalize some previous results of complex difference equations in one variable (see \cite[Theorem 9.2]{chiang-feng} and \cite[Theorem 6.2.3]{chen-book}.\par

\begin{theorem}\label{T7} Let $A_{0}, \ldots, A_{n}$ be meromorphic functions on $\mathbb{C}^{m}$ such that there exists an integer $k\in\{0, \ldots, n\}$ satisfying
\begin{eqnarray*}
\rho(A_{k})>\max\{\rho(A_{j}): 0\leq j\leq n, j\neq k\}, \,\,\mbox{and}\,\, \delta_{A_{k}}(\infty)>0.
\end{eqnarray*} If $f$ is a nontrivial meromorphic solution of linear partial difference equation
\begin{eqnarray}\label{E13}
A_{n}(z)f(z+c_{n})+\ldots+A_{1}(z)f(z+c_{1})+A_{0}(z)f(z)=0
\end{eqnarray} where $c_{1}, \ldots, c_{n}$ are distinct values of $\mathbb{C}^{m}\setminus\{0\},$ then we have $\rho(f)\geq\rho(A_{k})+1.$
\end{theorem}

\begin{proof} If $\rho(A_{k})=\infty,$ then we obviously get from \eqref{E13} that $f$ must be of infinite order. Without loss of generality, we assume $+\infty>\rho(A_{k})>0.$ In this case, it gives that $A_{k}$ must be transcendental. We find that there is nothing to do if $f$ is of infinity order. So, we may assume that $\rho(f)<+\infty.$ From the the equation \eqref{E13}, we get that the solution $f$ of \eqref{E13} can not be any nonzero rational function. Now we only need assume that $f$ is a transcendental meromorphic function with finite order.  The equation \eqref{E13} gives
\begin{eqnarray}\label{E14}
-A_{k}&=&A_{n}\frac{f(z+c_{n})}{f(z+c_{k})}+\ldots+A_{k+1}\frac{f(z+c_{k+1})}{f(z+c_{k})}\\\nonumber
&&+A_{k-1}\frac{f(z+c_{k-1})}{f(z+c_{k})}+\ldots+A_{0}\frac{f(z)}{f(z+c_{k})}.
\end{eqnarray}
Since $\delta:=\delta_{A_{k}}(\infty)>0,$ by the definition we get that \begin{eqnarray}\label{E16}N(r, A_{k})<(1-\frac{\delta}{2})T_{A_{k}}(r).\end{eqnarray} It yields by Theorem \ref{T1'} that \begin{eqnarray}\label{E15}
m(r, \frac{f(z+c_{j})}{f(z+c_{k})})=O(r^{\rho(f)-1+\varepsilon})
\end{eqnarray}for any $\varepsilon(>0),$ where $j\in \{0, 1, \ldots, n\}\setminus\{k\}$ and $c_{0}=0.$ Then from \eqref{E14}, \eqref{E16} and \eqref{E15}, we have
\begin{eqnarray}\label{E17}
\frac{\delta}{2}T(r, A_{k})&\leq& T(r, A_{k})-N(r, A_{k})\\\nonumber
&=&m(r, A_{k})\\\nonumber
&\leq&\sum_{0\leq j\leq n; j\neq k}m(r, A_{j})+\sum_{0\leq j\leq n; j\neq k}m(r, \frac{f(z+c_{j})}{f(z+k)})+O(1)\\\nonumber
&\leq& \sum_{0\leq j\leq n; j\neq k}T(r, A_{j})+O(r^{\rho(f)-1+\varepsilon}).
\end{eqnarray}
Set $$\max\{\rho(A_{j}): 0\leq j\leq n, j\neq k\}:=\sigma<\rho(A_{k}):=\rho$$  such that $\rho-\sigma>3\varepsilon>0.$
Then for the above $\varepsilon>0,$
$$T(r, A_{j})<r^{\sigma+\varepsilon}<r^{\rho-2\varepsilon}$$
holds for all $0\leq j\leq n, j\neq k.$ From the definition of order of $A_{k},$ there exists a sequence $\{r_{m}\}_{m=1}^{+\infty}$ (with $r_{m}\rightarrow\infty$ as $m\rightarrow\infty$)
such that $$T(r_{m}, A_{k})>r_{m}^{\rho-\varepsilon}$$ for sufficiently large $r_{m}.$ Hence, it follows  from \eqref{E17} that
\begin{eqnarray*}
\frac{\delta}{2}r_{m}^{\rho-\varepsilon}&\leq& (n-1) r_{m}^{\sigma+\varepsilon}+O(r_{m}^{\rho(f)-1+\varepsilon})\\
&\leq& (n-1) r_{m}^{\rho-2\varepsilon}+O(r_{m}^{\rho(f)-1+\varepsilon}),
\end{eqnarray*} and thus \begin{eqnarray*}
(\frac{\delta}{2}+o(1))r_{m}^{\rho-\varepsilon}&\leq& O(r_{m}^{\rho(f)-1+\varepsilon}).
\end{eqnarray*} This implies $\rho(f)\geq \rho+1=\rho(A_{k})+1.$
\end{proof}

Obvious, if the dominant coefficient $A_{k}$ is holomorphic, then $\delta_{A_{k}}(\infty)>0.$ Hence we get immediately the following corollary.\par

\begin{corollary} Let $A_{0}, \ldots, A_{n}$ be entire functions on $\mathbb{C}^{m}$ such that there exists an integer $k\in\{0, \ldots, n\}$ satisfying
\begin{eqnarray*}
\rho(A_{k})>\max\{\rho(A_{j}): 0\leq j\leq n, j\neq k\}.
\end{eqnarray*} If $f$ is a nontrivial entire solution of linear partial difference equation
\begin{eqnarray*}
A_{n}(z)f(z+c_{n})+\ldots+A_{1}(z)f(z+c_{1})+A_{0}(z)f(z)=0
\end{eqnarray*} where $c_{1}, \ldots, c_{n}$ are distinct values of $\mathbb{C}^{m}\setminus\{0\},$ then we have $\rho(f)\geq\rho(A_{k})+1.$
\end{corollary}

\begin{example}
Let $c\in\mathbb{C}.$ Then the entire function $w(z)=e^{z^{2}+z}$ of one variable is a solution of the linear partial difference equation
$$\frac{1}{e^{c^{2}+c}}w(z+c)-e^{2zc}w(z)=0.$$
Here $\rho(w)=2$ and $\rho(\frac{1}{e^{c^{2}+c}})=0$ and  $\rho(-e^{2zc})=1.$ This means that the conclusion $\rho(f)\geq \rho(A_{k})+1$ in Theorem \ref{T7} is sharp.
\end{example}

\begin{example}Let $c_{1}=(1, 0), c_{2}=(0, i)\in\mathbb{C}^{2}.$ Then $w(z)=e^{z_{1}^{2}+z_{2}^{2}}$ is an entire solution of   linear partial difference equation $$A_{2}(z)w(z+c_{2})+A_{1}(z)w(z+c_{1})+A_{0}w(z)=0,$$
that is $$A_{2}(z)w(z_{1}, z_{2}+i)+A_{1}(z)w(z_{1}+1, z_{2})+A_{0}w(z)=0,$$
where $A_{1}(z)\equiv 1,$ $A_{2}(z)=z_{1}+z_{2}$ and $A_{0}(z)=-\left((z_{1}+z_{2})e^{2z_{1}+1}+e^{2iz_{2}-1}\right).$
Here $\rho(w)=2$ and $\rho(A_{1})=\rho(A_{2})=0$ and  $\rho(A_{0})=1.$ This also means that the conclusion $\rho(f)\geq \rho(A_{k})+1$ in Theorem \ref{T7} is sharp.
\end{example}

\begin{example}Let $c_{1}=(1, i), c_{2}=(i, -1)\in\mathbb{C}^{2}.$ Then $w(z)=\frac{e^{z_{1}^{2}-2z_{2}^{2}}}{z_{1}+z_{2}}$ is a meromorphic solution of   linear partial difference equation $$A_{2}(z)w(z+c_{2})+A_{1}(z)w(z+c_{1})+A_{0}w(z)=0,$$
that is $$A_{2}(z)w(z_{1}+1, z_{2}+i)+A_{1}(z)w(z_{1}+i, z_{2}-1)+A_{0}w(z_{1}, z_{2})=0,$$
where $A_{1}(z)=\frac{1}{z_{1}+z_{2}},$ $A_{2}(z)=\frac{z_{1}+z_{2}+i-1}{z_{1}+z_{2}}$ and
$$A_{0}(z)=-\left(\frac{e^{2z_{1}-4z_{2}i+3}}{z_{1}+z_{2}+1+i}+e^{2iz_{1}+4z_{2}-3}\right).$$
Here $\rho(w)=2$ and $\rho(A_{1})=\rho(A_{2})=0$ and  $\rho(A_{0})=1.$ This also means that the conclusion $\rho(f)\geq \rho(A_{k})+1$ in Theorem \ref{T7} is sharp.
\end{example}

\begin{example}Let $c_{1}=(1, i), c_{2}=(i, 1)\in\mathbb{C}^{2}.$ Then $w(z)=e^{iz_{1}+z_{2}}-1$ is an entire solution of   linear partial difference equation $$A_{2}(z)w(z+c_{2})+A_{1}(z)w(z+c_{1})+A_{0}w(z)=0,$$
that is $$A_{2}(z)w(z_{1}+1, z_{2}+i)+A_{1}(z)w(z_{1}+i, z_{2}+1)+A_{0}w(z_{1}, z_{2})=0,$$
where $A_{1}(z)=A_{2}(z)\equiv 1$ and
$$A_{0}(z)=-1-\left(\frac{e^{iz_{1}+z_{2}+2i}-1}{e^{iz_{1}+z_{2}}-1}\right).$$
Here $\rho(w)=1$ and $\rho(A_{1})=\rho(A_{2})=0,$  $\rho(A_{0})=1$ and $\delta_{A_{0}}(\infty)=0.$ This implies that the assumption $\delta_{A_{k}}(\infty)>0$ in Theorem \ref{T7} is necessary.
\end{example}

\begin{question}It is obvious that $w(z)=\frac{e^{z_{1}+2z_{2}}}{z+{2z_{2}}}$ is a meromorphic solution of the partial difference equation $$A(z)w(z_{1}+1, z_{2}-1)+B(z)w(z_{1}, z_{2})=0,$$
where the coefficients $B(z)=-\frac{z_{1}+2z_{2}}{e}$ and $A(z)=z_{1}+2z_{2}-1$ are polynomials in $\mathbb{C}^{2}.$ Obviously $\rho(w)=1=\rho(A)+1=\rho(B)+1.$ Thus we ask what can be said for a general partial difference equation \eqref{E13} with all polynomial coefficients $A_{0}, \ldots, A_{n}$?
\end{question}

Since there is the model of the discrete or finite Poisson equation (see \cite{Cheng-book}) $$u_{i, j+1}+u_{i+1, j}+u_{i, j-1}+u_{i-1,j}-4u_{i,j}=g_{ij},$$ it is interesting to consider the following result on linear nonhomogeneous partial difference equations.\par

\begin{theorem}\label{T8} Let a meromorphic function $f$ on $\mathbb{C}^{m}$ be a solution of linear partial difference equation
\begin{eqnarray}\label{E3.12}
A_{n}(z)f(z+c_{n})+\ldots+A_{1}(z)f(z+c_{1})+A_{0}(z)f(z)=F(z),
\end{eqnarray} where meromorphic coefficients $A_{0}, \ldots, A_{n}, F(\not\equiv 0)$ on $\mathbb{C}^{m}$ are small functions with respect to $f,$ and $c_{1}, \ldots, c_{n}$ are distinct values of $\mathbb{C}^{m}\setminus\{0\}.$ If $\limsup_{r\rightarrow\infty}\frac{\log T(r, f)}{r}=0,$ then $\delta_{f}(0)=0.$
\end{theorem}

\begin{proof} Assume that the defect of zeros of $f$ satisfies $\delta_{f}(0)>0.$ Then we have
$$N(r, \frac{1}{f})<(1-\frac{\delta_{f}(0)}{2})T(r, f).$$ It follows from the equation \eqref{E3.12} that \begin{eqnarray*}\frac{1}{f}=\frac{1}{F}\left(A_{n}\frac{f(z+c_{n})}{f}+A_{n-1}\frac{f(z+c_{n-1})}{f}+\ldots+A_{1}\frac{f(z+c_{1})}{f}+A_{0}\right).\end{eqnarray*}
Under the assumption of $\limsup_{r\rightarrow\infty}\frac{\log T(r, f)}{r}=0,$ we get from the first main theorem and Theorem \ref{T1} that\begin{eqnarray*}m(r, \frac{1}{f})&\leq& m\left(r, \frac{1}{F}\right)+\sum_{j=1}^{n}m\left(r,\frac{f(z+c_{j})}{f}\right)+\sum_{j=0}^{n}m\left(r,  A_{j}\right)+O(1)\\&\leq&T(r, F)+\sum_{j=0}^{n}T(r, A_{j})+\sum_{j=1}^{n}m\left(r,\frac{f(z+c_{j})}{f}\right)+O(1)\\
 &=&o(T(r, f)).\end{eqnarray*} for $r\not\in E$ where $E$ is a set with $\overline{dens}E=0.$ This gives
\begin{eqnarray*} T(r, f)+O(1)&=&T(r, \frac{1}{f})=m(r, \frac{1}{f})+N(r, \frac{1}{f})\\
&\leq&N(r, \frac{1}{f})+o(T(r, f))\\
&\leq&(1-\frac{\delta_{f}(0)}{2})T(r, f)+o(T(r, f))
 \end{eqnarray*}  for all $r\not\in E.$ Therefore, we get
$$\delta_{f}(0)T(r,f)\leq o(T(r,f))$$
for all $r\not\in E$ where $E$ is a set with $\overline{dens}E=0.$ This is a contradiction.
\end{proof}

\begin{example} It is obvious that $f(z_{1}, z_{2})=z_{1}e^{z_{2}}$ with $\limsup_{r\rightarrow\infty}\frac{\log T(r, f)}{r}=0$ and $\delta_{f}(0)>0$ is an entire solution of partial difference equation
$$\frac{1}{e^{z_{2}}}f(z_{1}, z_{2}-1)+\frac{1}{z_{1}+1}f(z_{1}+1, z_{2})=e^{z_{2}}+\frac{z_{1}}{e}.$$ Here the coefficient $\frac{1}{z_{1}+1}$ is a small function with respect to $f,$ however the other coefficients $\frac{1}{e^{z_{2}}}$ and  $e^{z_{2}}+\frac{z_{1}}{e}$ are not. This means that it is necessary of  the assumption that the coefficients are small with respect to the solution in Theorem \ref{T8}.
\end{example}

\begin{example}
Let $c_{2}=(1,0)$ and $c_{1}=(0, -2i).$ Then $f(z)=e^{z_{1}^{2}+z_{2}z_{1}}$ with $\limsup_{r\rightarrow\infty}\frac{\log T(r, f)}{r}=0$ and $\delta_{f}(0)=1$ is an entire solution of partial difference equation
$$\frac{1}{e^{1+2z_{1}}}f(z_{1}+1, z_{2})-e^{2iz_{2}}f(z_{1}, z_{2}-2i)=0.$$ Here the coefficient $\frac{1}{e^{1+2z_{1}}}$ and $-e^{2iz_{2}}$ are small functions with respect to the solution $f.$  This implies that the coefficient $F$ in Theorem \ref{T8} can not be identical to zero.
\end{example}

\subsection{Difference analogues of Tumura-Clunie theorem in several complex variables}\noindent

Now, we will extend difference version of Tumura-Clunie theorem from one variable to several variables. The Clunie lemma \cite{clunie} for meromorphic functions of one variable in Nevanlinna theory has been a powerful tool of studying complex differential equations and related fields, particularly the lemma has been used to investigate the value distribution of certain differential polynomials; see \cite{clunie} for the original versions of these results, as well as \cite{hayman, laine}. A slightly more general version of the Clunie lemma can be found in \cite[pp.~218--220]{he-xiao}; see also \cite[Lemma~2.4.5]{laine}. In 2007, the additional assumptions in the He-Xiao version  of the Clunie lemma have been removed by Yang and Ye in \cite[Theorem 1]{yang-ye}. A generalized Clunie lemma for meromorphic functions of several complex variables was proved in \cite{libaoqin}; for some special cases refer to see \cite{hu-li-yang-1, hu-yang-1}. Recently, Hu and Yang \cite{hu-yang} extended the classical Tumura-Clunie theorem (\cite[Theorem 3.9]{hayman} and \cite{mues-steinmetz}) for meromorphic functions of one variable to that of meromorphic functions of several complex variables.  \par

We prove a difference counterpart of the Hu-Yang's version \cite{hu-yang} of Tumura-Clunie theorem in several complex variables as follows, which in fact generalize and extend the corresponding result of one variable due to Laine and Yang \cite[Theorem~1]{laine-yang-2} modified later by Chen, Huang and Zheng \cite{chen-huang-zheng} (see also\cite[Theorem 4.3.4]{chen-book}). Set a difference polynomial of several complex variables
\begin{eqnarray}\label{E-7.4}
G(z, f)=\sum_{\lambda\in J} b_{\lambda}(z)\prod_{j=1}^{\tau_{\lambda}}f(z+q_{\lambda, j})^{\mu_{\lambda,j}},
\end{eqnarray}where $\max_{\lambda\in J}\sum_{j=1}^{\tau_{\lambda}}\mu_{\lambda,j}=n,$ and $q_{\lambda,j}\not=0$ for at least one of the constants $q_{\lambda,j}$. Moreover, we assume that the coefficients in \eqref{E-7.4} are meromorphic functions on~$\mathbb{C}^{m}$ and small with respect to the function $f$, which is meromorphic on $\mathbb{C}^{m}.$\par

\begin{theorem}\label{T-8.2}
Let $f$ be a meromorphic function on $\mathbb{C}^{m}$ with $$\limsup_{r\rightarrow\infty}\frac{\log T(r, f)}{r}=0,$$ such that
\begin{eqnarray}\label{E-7.5}
N\left(r, \frac{1}{f}\right)+N(r, f)=o(T(r, f)).
\end{eqnarray}  Suppose that the difference polynomial \eqref{E-7.4} of $f(z)$ and its shifts is  of maximal total degree $n.$ If $G$ also satisfies
\begin{eqnarray}\label{E-7.7}
\sum_{\lambda\in J_{n-1}} b_{\lambda}(z)\prod_{j=1}^{\tau_{\lambda}}f(z+q_{\lambda, j})^{\mu_{\lambda,j}}\not\equiv 0,
\end{eqnarray}
where $J_{n-1}=\{\lambda\in J: \sum_{j=1}^{\tau_{\lambda}}\mu_{\lambda,j}=n-1\},$ then $G$ must
satisfy
\begin{eqnarray*}
N\left(r, \frac{1}{G}\right)\neq o(T(r, f)).
\end{eqnarray*}
\end{theorem}

For the proof of Theorem \ref{T-8.2}, we first need the Tumura-Clunie theorem of several complex variables due to Hu and Yang.\par

\begin{lemma}\label{L-8.1}\cite[Theorem 2.1]{hu-yang} Suppose that $f$ is meromorphic and not constant in $\mathbb{C}^{m},$ that $$g=f^{n}+P_{n-1}(f),$$ where $P_{n-1}(f)$ is a differential polynomial of degree at most $n-1$ in $f,$ and that $$N(r, f)+N\left(r, \frac{1}{g}\right)=o(T(r, f)).$$ Then $$g=\left(f+\frac{\alpha}{n}\right)^{n},$$ where $\alpha$ is a meromorphic function in $\mathbb{C}^{m}$, small with respect to $f$, and determined by the terms of degree $n-1$ in $P_{n-1}(f)$ and by $g.$
\end{lemma}

We also need the second main theorem for meromorphic functions with small function targets on $\mathbb{C}^{m}.$  It is mentioned in \cite[Theorem 2.1]{chen-yan} that the conclusion is easily extended from the second main theorem for small function targets due to Yamanoi~\cite{yamanoi} by the standard process of averaging over the complex lines in $\mathbb{C}^{m}.$\par

\begin{lemma}\label{L4} Let $f$ be a meromorphic function on $\mathbb{C}^{m},$ and $a_{1}, \ldots a_{q}$ be distinct meromorphic functions "small" with respect to $f.$ Then we have
$$(q-2)T(r, f)\leq \sum_{j=1}^{q}\overline{N}\left(r, \frac{1}{f-a_{j}}\right)+o(T(r, f))$$ for all $r\not\in F,$ where $F$ is a set of finite Lebegue logarithmic measure.\end{lemma}

\begin{proof}[Proof of Theorem \ref{T-8.2}] Suppose that the conclusion is not true, that is,
$$N\left(r, \frac{1}{G}\right)= o(T(r, f)).$$ To prove this theorem, we propose to follow the idea in the proof of \cite[Theorem~1]{laine-yang-2}. Since the difference polynomial \eqref{E-7.4} of $f(z)$ and its shifts is of maximal total degree $n,$ we get
\begin{eqnarray*}
G(z, f)&=&\sum_{\lambda\in J} b_{\lambda}(z)\prod_{j=1}^{\tau_{\lambda}}f(z+q_{\lambda, j})^{\mu_{\lambda,j}}\\
&=&\sum_{\lambda\in J} b_{\lambda}(z)\prod_{j=1}^{\tau_{\lambda}}\left[\left(\frac{f(z+q_{\lambda, j})}{f(z)}\right)^{\mu_{\lambda,j}}\cdot f(z)^{\mu_{\lambda,j}}\right]\\
&:=&\sum_{j=0}^{n}\tilde{b}_{j}(z)f(z)^{j},
\end{eqnarray*}  where each of the coefficients $\tilde{b}_{j}(z)$ $(j=1, \ldots, n)$ is the sum of finitely many terms of type
$$b_{\lambda}(z)\left(\frac{f(z+q_{\lambda, j})}{f(z)}\right)^{\mu_{\lambda,j}}.$$ It yields \begin{eqnarray*}
\frac{G(z, f)}{\tilde{b}_{n}(z)}=f(z)^{n}+\sum_{j=0}^{n-1}\frac{\tilde{b}_{j}(z)}{\tilde{b}_{n}(z)}f(z)^{j}.
\end{eqnarray*} In terms of the assumption \eqref{E-7.7}, we have  $\sum_{j=0}^{n-1}\frac{\tilde{b}_{j}(z)}{\tilde{b}_{n}(z)}f(z)^{j}\not\equiv 0.$\par

Note that all the coefficient functions $b_{\lambda}(z)$ $(\lambda\in J)$ are small with respect to $f.$ Then by Theorem \ref{T1} we get that for all $j=1,\ldots, n,$
$$m(r, \tilde{b}_{j})=o(T(r, f))$$
holds for all $r\not\in E$ with $\overline{dens} E=0.$ Moreover, by the assumption \eqref{E-7.5} and Lemma \ref{L1} we have
$$N(r, \tilde{b}_{j})=o(T(r, f)),$$
and thus $$T(r,\tilde{b}_{j})=o(T(r, f)), \quad j\in\{0, 1, \ldots, n\}$$
and
\begin{equation*} N\left(r, \frac{1}{\frac{G(z, f)}{\tilde{b}_{n}(z)}}\right)= o(T(r, f)).\end{equation*}
for all $r\not\in E.$  Hence by Lemma \ref{L-8.1} we may write
\begin{eqnarray*}
\frac{G(z, f)}{\tilde{b}_{n}(z)}=\left(f(z)+\frac{\alpha(z)}{n}\right)^{n},
\end{eqnarray*} where $\alpha\not\equiv 0$ and  $T(r, \alpha)=o(T(r, f)).$ This implies that
\begin{equation}\label{E-7.6}N\left(r, \frac{1}{f(z)+\frac{\alpha(z)}{n}}\right)=o(T(r, f)).\end{equation} Together with \eqref{E-7.5} and \eqref{E-7.6}, it follows from Lemma \ref{L4} that
$$T(r, f)\leq N\left(r, \frac{1}{f}\right)+N(r, f)+N\left(r, \frac{1}{f(z)+\frac{\alpha(z)}{n}}\right)+o(T(r, f))=o(T(r, f))$$ for all $r\not\in (E\cup F),$ where $F$ is a set of finite Lebegue logarithmic measure. Hence we get a contradiction.
\end{proof}

Moreover, we improve and extend Laine-Yang's difference analogue of Clunie theorem in one variable \cite{laine-yang} to high dimension by using Theorem \ref{T1}. Define complex partial difference polynomials as follows
\begin{eqnarray}\label{E-7.1}
P(z,w)=\sum_{\lambda\in I}a_{\lambda}(z)w(z)^{l_{\lambda_{0}}}w(z+q_{\lambda_{1}})^{l_{\lambda_{1}}}\cdots w(z+q_{\lambda_{i}})^{l_{\lambda_{i}}},\end{eqnarray}
\begin{eqnarray}\label{E-7.2} Q(z,w)=\sum_{\mu\in J}b_{\mu}(z)w(z)^{l_{\mu_{0}}}w(z+q_{\mu_{1}})^{l_{\mu_{1}}}\cdots w(z+q_{\mu_{j}})^{l_{\mu_{j}}},\end{eqnarray}
\begin{eqnarray}\label{E-7.3} U(z,w)=\sum_{\nu\in K}c_{\nu}(z)w(z)^{l_{\nu_{0}}}w(z+q_{\nu_{1}})^{l_{\nu_{1}}}\cdots w(z+q_{\nu_{k}})^{l_{\nu_{k}}},
\end{eqnarray} where all coefficients $a_{\lambda}(z),$ $b_{\mu}(z)$ and $c_{\nu}(z)$ are small functions with respect to the function $w(z)$ meromorphic on $\mathbb{C}^{m},$ $I, J, K$ are three finite sets of multi-indices, and $q_{s}\in\mathbb{C}^{m}\setminus\{0\},$ $(s\in\{\lambda_{1}, \ldots, \lambda_{i}, \mu_{1}, \ldots, \mu_{j}, \nu_{1}, \ldots, \nu_{k}\}).$ Since the proof is closely similar as in \cite{laine-yang}, we omit it here. \par

\begin{theorem}\label{T-8.1}
Let $w$ be a nonconstant meromorphic function on $\mathbb{C}^{m}$ with $$\limsup_{r\rightarrow\infty}\frac{\log T(r, w)}{r}=0,$$ and let $P(z,w), Q(z, w),$ and $U(z,w)$ be complex partial difference polynomials as \eqref{E-7.1}, \eqref{E-7.2} and \eqref{E-7.3} satisfying a complex partial difference equation of the form \begin{eqnarray}
U(z,w)P(z,w)=Q(z,w).
\end{eqnarray}
Assume that the total degree of $U(z, w)$ is equal to $n,$ and the total degree of $Q(z,w)$ is less than or equal to $n,$ and that $U(z,w)$ contains just one term of maximal total degree in $w(z)$ and its shifts.  Then we have
\begin{equation*}
m(r, P(z,w))=o(T(r, w))
\end{equation*} for all $r\not\in E$ where $E$ is a set with $\overline{dens}E=0.$
\end{theorem}

\subsection{Improvement of Korhonen's result}\noindent

Finally, by applying Theorem \ref{T1}, Theorem \ref{T5} and Valion-Mohon'ko theorem in several complex variables \cite[Theorem 3.4]{hu-yang-1} into the following equation \eqref{E12}, it is easy to follow that $$T(r, w)=\deg_{w}(R)T(r, w)+o(T(r, w))$$ for all $r\not\in E$ with $\overline{dens}E=0.$ We restate \cite[Theorem 4.1]{korhonen-1} as follows.\par

\begin{theorem}
Let $c\in\mathbb{C}^{n}\setminus\{0\}.$ If the difference equation \begin{equation}\label{E12} w(z+c)=R(z, w(z)),\end{equation} where $R(z, u)$ is rational in $u$ having meromorphic coefficients in $\mathbb{C}^{n},$ has an admissible meromorphic solution $w$ on $\mathbb{C}^{n}$ with $$\limsup_{r\rightarrow\infty}\frac{\log T(r, w)}{r}=0,$$ then the degree $\deg_{w}(R)$ of $R(z, w(z))$ is equal to one.
\end{theorem}

\section{Concluding remark}

The partial difference equations can be regarded as discrete analogue of partial differential equations. Although partial difference equations  appear earlier than partial differential equations, the former equations have not drawn as much attention as their continuous counterparts. In this paper as we shown, in the viewpoint of Nevanlinna theory in complex analysis, the improvements of logarithmic difference lemma for meromorphic functions in several complex variables, the relations $N(r, f(z+c)\sim N(r, f(z))$ and $T(r, f(z+c))\sim T(r,f (z))$ are obtained. Then we focus basically on some typical partial difference equations  such as linear partial difference equations with function coefficients coming from the models of discrete heat equation, discrete Laplace equation, nonsymmetric partial difference functional equation, and discrete Poisson equation and others, the nonlinear partial difference equations coming from discrete potential KdV equation and the Fermat equation, and partial difference equation concerning the Tumura-Clunie theorem. Of course, it is impossible to study completely and systematically the meromorphic solutions of partial difference equations in one paper. This paper is an attempt to do this by making use of the logarithmic difference lemma of several complex variables in Nevanlinna theory. Much work on meromorphic solutions of partial difference equations remains to be done in the near future. \par\vskip 10pt

\noindent{\bf Acknowledgement.} The authors are very grateful to the anonymous reviewer for important suggestions and comments to improve this paper. The authors would like to thank to Professor Risto Korhonen for introducing the reference \cite{tremblay-grammaticos-ramani} about the discrete KdV equations, and also thank to Professor Jianhua Zheng for discussing partly of the proof of Theorem \ref{T1''}.

\bibliographystyle{amsplain}

\end{document}